\documentclass[11pt,a4paper,reqno]{amsart}
\usepackage{amsmath,amssymb,amsfonts,amsthm}
\usepackage[francais,english]{babel}
\usepackage{a4wide}
\usepackage{color}
\usepackage{pdfsync}
\usepackage{graphicx,subfigure}
\usepackage{hyperref}

\theoremstyle{plain}

\newtheorem{theorem}{Theorem}[section]
\newtheorem{lemma}[theorem]{Lemma}
\newtheorem{proposition}[theorem]{Proposition}
\newtheorem{corollary}[theorem]{Corollary}

\newcommand{\vertiii}[1]{{\left\vert\kern-0.25ex\left\vert\kern-0.25ex\left\vert #1 \right\vert\kern-0.25ex\right\vert\kern-0.25ex\right\vert}}% defines an operator norm with triple verical lines

\theoremstyle{remark}

\newtheorem*{remark}{Remark}

\numberwithin{equation}{section}
\makeatletter

\@addtoreset{equation}{section}
\makeatother
\newcommand{\Z}{{\mathbb Z}}

\newcommand{\R}{{\mathbb R}}

\def\curl{{\rm curl }\,}
\def\div{{\rm div }\,}
\def\d{\, {\rm d }}

\title[The vortex method]{The vortex method for 2D ideal flows \\ in the exterior of a disk}
\author[D. Ars\'enio, E. Dormy \& C. Lacave]{Diogo Ars\'enio, Emmanuel Dormy  \& Christophe Lacave}

\address[D. Ars\'enio]{CNRS, Institut de Mathématiques de Jussieu-Paris Rive Gauche, UMR7586, Univ Paris Diderot, Sorbonne Paris Cité, Sorbonne Universités, UPMC Univ Paris 06, F-75013 Paris, France.} 
\email{diogo.arsenio@imj-prg.fr}

\address[E. Dormy]{CNRS \& MAG (ENS/IPGP), D\'epartement de Physique, 24 rue Lhomond, 75005 Paris, France.}
\email{dormy@phys.ens.fr}

\address[C. Lacave]{Univ Paris Diderot, Sorbonne Paris Cité, Institut de Mathématiques de Jussieu-Paris Rive Gauche, UMR 7586, CNRS, Sorbonne Universités, UPMC Univ Paris 06, F-75013, Paris, France.}
\email{christophe.lacave@imj-prg.fr}
\date{October 3, 2014.}

\begin{document}
\maketitle
\begin{abstract}
The vortex method is a common numerical and theoretical approach used to implement the motion of an ideal flow, in which the vorticity is approximated by a sum of point vortices, so that the Euler equations read as a system of ordinary differential equations. Such a method is well justified in the full plane, thanks to the explicit representation formulas of Biot and Savart. In an exterior domain, we also replace the impermeable boundary by a collection of point vortices generating the circulation around the obstacle. The density of these point vortices is chosen in order that the flow remains tangent at midpoints between adjacent vortices. In this work, we provide a rigorous justification for this method in exterior domains. One of the main mathematical difficulties being that the Biot-Savart kernel defines a singular integral operator when restricted to a curve. For simplicity and clarity, we only treat the case of the unit disk in the plane approximated by a uniformly distributed mesh of point vortices. The complete and general version of our work is available in \cite{ADL}.
\end{abstract}
\selectlanguage{francais}
\begin{abstract}
La m\'ethode des vortex est une approche th\'eorique et num\'erique couramment utilis\'ee afin d'impl\'ementer le mouvement d'un fluide parfait, dans laquelle le tourbillon est approch\'e par une somme de points vortex, de sorte que les \'equations d'Euler se r\'e\'ecrivent comme un syst\`eme d'\'equations diff\'erentielles ordinaires. Une telle m\'ethode est rigoureusement justifi\'ee dans le plan complet, gr\^ace aux formules explicites de Biot et Savart. Dans un domaine ext\'erieur, nous rempla\c cons \'egalement le bord imperm\'eable par une collection de points vortex, g\'en\'erant une circulation autour de l'obstacle. La densit\'e de ces points est choisie de sorte que le flot demeure tangent au bord sur certains points interm\'ediaires aux paires de tourbillons adjacents sur le bord. Dans ce travail, nous proposons une justification rigoureuse de cette m\'ethode dans des domaines ext\'erieurs. L'une des principales difficult\'es math\'ematiques \'etant que le noyau de Biot-Savart d\'efinit un op\'erateur int\'egral singulier lorsqu'il est restreint \`a une courbe. Par souci de simplicit\'e et de clart\'e, nous traitons seulement le cas du disque unit\'e dans le plan, approch\'e par un maillage de points uniform\'ement r\'epartis. La version compl\`ete et g\'en\'erale de notre travail est disponible en \cite{ADL}.
\\
\\ Mots cl\'es: \'equations d'Euler, probl\`eme elliptique dans des domaines ext\'erieurs, transform\'ee de Hilbert et transform\'ee de Hilbert discr\`ete.
\end{abstract}
\selectlanguage{english}

\section{Introduction}

Numerical methods describing the evolution of a flow have many practical interests in engineering and applications. It is therefore important to justify that given methods provide good approximations of analytic solutions. The goal of this proceeding is to validate the vortex method in exterior domains  for the two-dimensional Euler equations.

\subsection{The Euler equations in exterior domains}

The motion of an incompressible ideal fluid filling a domain $\Omega \subset \R^2$ is governed by the Euler equations:
\begin{equation}\label{Euler}
	\left\{
	\begin{array}{lcl}
		\partial_{t} u + u\cdot \nabla u +\nabla p=0 & \text{in} &(0,\infty)\times \Omega, \\
		\div u  =0 & \text{in} &[0,\infty)\times \Omega, \\
		u \cdot n  = 0 & \text{on} &[0,\infty)\times \partial \Omega, \\
		u(0,\cdot)  = u_0 & \text{in} & \Omega, \\
	\end{array}
	\right.
\end{equation}
where $u=(u_{1}(t,x_{1},x_{2}),u_{2}(t,x_{1},x_{2}))$ is the velocity, $p=p(t,x_{1},x_{2})$ the pressure and $n$ the unit inward normal vector.
There is an impressive literature about the study of this system, first for physical motivations and second because it provides elegant mathematical problems at the boundaries of elliptic theory, dynamical systems, convex geometry and partial differential equations. The richness of these equations is due to the role of the vorticity:
\[
\omega(t,x):= \curl u(t,x) = \partial_{1} u_{2} - \partial_{2} u_{1}.
\]
Indeed, taking the curl of the momentum equation, we note that this quantity satisfies a transport equation:
\begin{equation}\label{Euler vort}
		\partial_{t} \omega + u\cdot \nabla \omega =0  \quad \text{in}\quad (0,\infty)\times \Omega.
\end{equation}
From this form, we deduce many conservation properties which allow to establish the wellposedness of the Euler equations in several different settings (standard references can be found in  \cite{GVL2,MajdaBertozzi}). Therefore, one of the key steps in the analysis of \eqref{Euler} consists in reconstructing the velocity $u$ from the vorticity $\omega$ by solving the following elliptic problem:
\begin{equation}\label{elliptic}
	\left\{
	\begin{array}{lcl}
		\div u  =0 & \text{in}& \Omega, \\
		\curl u  =\omega & \text{in}& \Omega, \\
		u \cdot n  = 0 & \text{on}&\partial\Omega, \\
		u \rightarrow 0 & \text{as}& x\rightarrow\infty, \\
	\end{array}
	\right.
\end{equation}
where $\omega\in C^{0,\alpha}_c\left(\Omega\right)$, for some $0<\alpha\leq 1$.

In the case of the full plane $\Omega=\R^2$, any solution of
\begin{equation}\label{system BS}
		\div u  =0  \text{ in } \mathbb{R}^2, \quad
		\curl u  =\omega \text{ in } \mathbb{R}^2, \quad
		u  \rightarrow 0  \text{ as } x\rightarrow\infty,
\end{equation}
satisfies
\begin{equation*}
	\Delta u = \nabla^\perp \omega \quad\text{in }\mathbb{R}^2,
\end{equation*}
which easily yields
\begin{equation*}
	u=K_{\mathbb{R}^2}[\omega]=\mathcal{F}^{-1}\frac{i\xi^\perp}{|\xi|^2}\mathcal{F}\omega.
\end{equation*}
Here, the superscript $\perp$ denotes the rotation by $\pi/2$, that is $(x_{1},x_{2})^\perp =(-x_{2},x_{1})$. It follows, employing standard results on Fourier multipliers, that $K_{\mathbb{R}^2}$ has bounded extensions from $L^p$ to $\dot W^{1,p}$, for any $1<p<\infty$. Furthermore, writing $\Phi(x)=-\frac 1{2\pi}\log|x|$ the fundamental solution of the Laplacian in $\mathbb{R}^2$, it holds that (see e.g. \cite{gilbarg})
\begin{equation}\label{BS R2}
	u=K_{\mathbb{R}^2}[\omega]=- \Phi*\left(\nabla^\perp\omega\right)=-\nabla^\perp \left(\Phi*\omega\right)=\frac 1{2\pi}\int_{\mathbb{R}^2} \frac{(x-y)^\perp}{|x-y|^2}\omega(y)dy
	\in C^1\left(\mathbb{R}^2\right).
\end{equation}
We refer to \cite[p. 249]{courant} for a justification of the $C^1$-regularity of $K_{\mathbb{R}^2}[\omega]$.

When $\Omega=\{ x\in \R^2,\ |x|>1\}$ is the exterior of the unit disk, there are an infinite number of solution of \eqref{elliptic}, because of the harmonic vector field:
\[
H(x) =\frac1{2\pi}\frac{x^\perp}{|x|^2},
\]
which verifies
\begin{equation*}
		\div H  =0  \text{ in } \Omega, \quad
		\curl H  =0 \text{ in } \Omega, \quad
		H\cdot n  =0 \text{ on } \partial\Omega, \quad
		H  \rightarrow 0  \text{ as } x\rightarrow\infty.
\end{equation*}
Thus, in order to reconstruct uniquely the velocity in terms of the vorticity, the standard idea consists in prescribing the circulation:
\[
\oint_{\partial\Omega} u \cdot \tau \, ds = \gamma,
\]
where $\gamma\in \R$ and $\tau:= n^\perp$ is the tangent vector to $\partial\Omega$. This constraint is natural because Kelvin's theorem implies then that the circulation of $u$ around an obstacle is a conserved quantity for the Euler equations. With this additional condition, it holds now true that there exists a unique solution $u$ of
\begin{equation}\label{elliptic2}
	\left\{
	\begin{array}{lcl}
		\div u  =0 & \text{in}& \Omega, \\
		\curl u  =\omega & \text{in}& \Omega, \\
		u \cdot n = 0 & \text{on}& \partial\Omega, \\
		u  \rightarrow 0 &\text{as}& x\rightarrow\infty, \\
		\oint_{\partial\Omega} u \cdot \tau \d s  = \gamma, 
	\end{array}
	\right.
\end{equation}
where $\omega\in C^{0,\alpha}_c\left(\Omega\right)$, for some $0<\alpha\leq 1$, and $\gamma\in\mathbb{R}$ (see e.g. \cite[Prop. 2.1]{ILL}).

To solve this elliptic problem, we introduce the Green function with Dirichlet boundary condition $G_{\Omega}:\ \Omega\times\Omega\to \R$ as a function verifying:
\begin{equation*}
	\begin{aligned}
		G_{\Omega}(x,y) & =G_{\Omega}(y,x) &&\forall (x,y)\in \Omega^2, \\
		\Delta_{x} G_{\Omega}(x,y) & = \delta(x-y)\ &&\forall (x,y)\in \Omega^2, \\
		G_{\Omega}(x,y) & =0 &&\forall (x,y)\in \partial\Omega \times \Omega,
	\end{aligned}
\end{equation*}
where $\delta$ denotes the Dirac function centered at the origin. In the case of the exterior of the unit disk $D:= \overline{B(0,1)}$, we have an explicit formula:
\[
G_{\Omega}(x,y) = \frac1{2\pi} \ln \frac{|x-y|}{|x-y^*| |y|},
\]
with the notation $y^*=\frac{y}{|y|^2}$, for any $y\in\mathbb{R}^2\setminus\left\{0\right\}$. This expression allows us to write explicitly the solution of \eqref{elliptic2} (for all details, we refer e.g. to \cite{ILL}):
\begin{equation}\label{BS exterior}
\begin{split}
u(x) &= K_{\Omega}[\omega](x) + \alpha H(x) := \int_{\Omega} \nabla^\perp_{x} G_{\Omega}(x,y) \omega(y)\, dy +  \alpha H(x)\\
& = \frac1{2\pi} \int_{\Omega} \Big( \frac{x-y}{|x-y|^2} - \frac{x-y^*}{|x-y^*|^2} \Big)^{\perp}\omega(y)\, dy + \frac{\alpha}{2\pi} \frac{x^\perp}{|x|^2}\\
& = \frac1{2\pi} \int_{\mathbb{R}^2} \frac{\left(x-y\right)^\perp}{|x-y|^2} \left(\omega(y)-\frac{1}{|y|^4}\omega(y^*)\right)\, dy + \frac{\alpha}{2\pi} \frac{x^\perp}{|x|^2}
\in C^1\left(\overline\Omega\right),
\end{split}\end{equation}
where we have set
\begin{equation*}
\alpha = \gamma + \int_{\Omega}\omega(y) \, dy.
\end{equation*}
Note that the total mass of the vorticity is also a conserved quantity of incompressible ideal two-dimensional flows.

In conclusion, the Euler equations around the obstacle $D$ can be seen as the transport of the vorticity \eqref{Euler vort} by the velocity field $u$ defined by \eqref{BS exterior}. This property conveniently allows the use of various mathematical theories. It is therefore crucial to develop efficient and robust methods to rebuild the velocity field $u$ from the vorticity $\omega$ or an approximation of it. In particular, for the sake of applications, we are now going to focus on the theoretical and numerical approximation of \eqref{BS exterior}.

\subsection{The vortex method}

In the full plane $\mathbb{R}^2$, when the initial vorticity is close to be concentrated at $N$ given points $\left\{x_i^0\right\}_{i=1}^N\subset\mathbb{R}^2$, i.e. $\omega(t=0)\sim\sum_{i=1}^N \gamma_{i} \delta_{x_{i}^0}$ in some suitable sense, Marchioro and Pulvirenti \cite{MP91} have shown that the corresponding solution of the Euler equations in the full plane has a vorticity which remains close to a combination of Dirac masses $\omega(t)\sim\sum_{i=1}^N \gamma_{i} \delta_{x_{i}(t)}$ (in some suitable sense) where the centers $\{ x_{i} \}_{i=1}^N$ verify a system of ODE's, called the point vortex system:
\begin{equation}\label{point vortex 0}
\left\{ \begin{split}
\dot{x}_{i}(t) &= \frac1{2\pi} \sum_{j\neq i} \gamma_{j} \frac{(x_{i}(t)-x_{j}(t))^\perp}{|x_{i}(t)-x_{j}(t)|^2},\\
x_{i}(0)&=x_{i}^0.
\end{split}
\right.
\end{equation}
Here, the point vortex $\gamma_{i} \delta_{x_{i}(t)}$ moves under the velocity field produced by the other point vortices.

It turns out that this Lagrangian formulation is much easier to handle numerically than the Eulerian formulation \eqref{Euler vort}. Indeed, standard numerical methods on \eqref{Euler vort} generate an ``inherent numerical viscosity'' and some quantities which should be conserved instead decrease (see e.g. \cite{Hirsch,Toro}). Actually, smoothing the Biot-Savart kernel by mollifying $\frac{x^\perp}{|x|^2}$ in \eqref{point vortex 0} gives a more stable system, called the vortex-blob method (i.e. approximation of the vorticity by Dirac masses and regularization of the kernel). The stability and the convergence as $N\rightarrow\infty$ of the vortex-blob and point vortex methods have been extensively studied: in \cite{CCM88} for the vortex-blob method when the initial vorticity is bounded, in \cite{GHL90} for the point vortex method for smooth initial data and in \cite{Schochet} for both methods and for weak solutions as e.g. a vortex sheet (see also the textbook \cite{CK00}).

However, all these works use the explicit formula of the Biot-Savart law in the full plane \eqref{BS R2} where the flow $\frac{(x-x_{i})^\perp}{2\pi |x-x_{i}|^2}$ is identified with $K_{\R^2}[\delta_{x_{i}}]$. In an exterior domain, the Biot-Savart law is much more complicated. A possible approach could be to use the  explicit formula \eqref{BS exterior} in order to adapt the previous vortex methods. But such an approach would only be useful in the exterior of the disk. Indeed, if we consider that $\Omega$ is the exterior of a compact, simply connected subset of $\R^2$, we can implicitly adapt formula \eqref{BS exterior} thanks to conformal mappings, which has some theoretical interest, but this approach yields serious practical difficulties, for there are very few explicit Riemann mappings available.

Our alternative strategy consists in approximating the impermeable boundary of the exterior domain by a collection of point vortices $\sum_{i=1}^N \frac{\gamma_{i}^N(t)}N \delta_{x_{i}}$, where the vortex positions $\{x_{i}\}_{i=1}^N$ are fixed but the density of points $\{\gamma_{i}^N\}_{i=1}^N$ now evolves with time and is chosen in order that the resulting velocity field remains tangent at midpoints on the boundary between the $x_i$'s. Note that this approach appears sometimes in physics and engineering books (see e.g. \cite{BL93,GZ09}).

To this end, we introduce $u_{P}$ the solution of \eqref{system BS} in the full plane, which is explicitly given by \eqref{BS R2}:
\begin{equation}\label{uP}
	u_{P}:=K_{\R^2}[\omega]\in C^1\left(\mathbb{R}^2\right)\subset C^1\left(\overline\Omega\right),
\end{equation}
and the remainder velocity field $u_{R}$ defined by:
\begin{equation}\label{uR}
	u_R:=u-u_P\in C^0\left(\overline\Omega\right)\cap C^1\left(\Omega\right).
\end{equation}
As $\omega$ is compactly supported in $\Omega$ we get by the Stokes formula that $\oint_{\partial\Omega} u_P \cdot \tau \d s =\int_{B(0,1)}\curl u_P = \int_{B(0,1)} \omega=0$. Hence, it is readily seen that $u_R$ solves
\begin{equation}\label{eq uR}
	\left\{
	\begin{array}{lcl}
		\div u_R  =0 & \text{in}& \Omega, \\
		\curl u_R  =0 & \text{in}& \Omega, \\
		u_R \cdot n  = -u_P\cdot n & \text{on }&\partial\Omega, \\
		u_R  \rightarrow 0 & \text{as}& x\rightarrow\infty, \\
		\oint_{\partial\Omega} u_R \cdot \tau \d s  = \gamma. &&
	\end{array}
	\right.
\end{equation}
In particular, $u_R$ is harmonic in $\Omega$ and therefore it is smooth in $\Omega$, i.e. $u_R\in C^\infty\left(\Omega\right)$ (see \cite[Corollary 8.11]{gilbarg} or \cite{gray}).

The vortex method for the exterior domain $\Omega$ is essentially an approximation procedure of $u_R$ by point vortices on $\partial \Omega$.

Thus, let now $(x_{1}^N, x_{2}^N,\dots , x_{N}^N)$ be the positions of $N$ distinct point vortices on the boundary $\partial\Omega$. In the case of the disk, $
	\partial\Omega=\partial B(0,1)
	=
	\left\{\left(\cos\theta,\sin\theta\right)\in\mathbb{R}^2\ :\ \theta\in[0,2\pi)\right\}
$, we consider
\begin{equation}\label{xi}
0=\theta_{1}^N< \theta_{2}^N<\dots < \theta_{N}^N<2\pi \text{ such that }x_{i}^N=(\cos\theta_{i}^N,\sin \theta_{i}^N).
\end{equation}
We further introduce some intermediate points on the boundary, for each $i=1,\dots,N-1$:
\begin{equation}\label{tildexi}
\tilde\theta_{i}^N \in (\theta_{i}^N,\theta_{i+1}^N),\quad \tilde x_{i}^N:=(\cos\tilde\theta_{i}^N,\sin \tilde\theta_{i}^N).
\end{equation}
The method consists in approximating the solution $u_R$ to \eqref{eq uR} by a suitable flow
\begin{equation}\label{approx}
u_{\rm app}^N(x):=\frac1{2\pi} \sum_{j=1}^N \frac{\gamma_{j}^N}N \frac{( x - x_{j}^N)^\perp}{|x - x_{j}^N|^2} = K_{\R^2}\Big[ \sum_{j=1}^N \frac{\gamma_{j}^N}N \delta_{x_{j}^N} \Big],
\end{equation}
whose vorticity is precisely made of $N$ point vortices with density $\left\{\frac{\gamma_i^N}{N}\right\}_{i=1}^N$ on the boundary $\partial\Omega$.

It is to be emphasized that this approximation is consistent with and motivated by the physical idea that the circulation around the obstacle (here, the unit disk $B(0,1)$) is created by a collection of vortices on the boundary of the obstacle, i.e. a vortex sheet on the boundary.

However, it is a priori not obvious that such a flow $u_{\rm app}^N$ can be made a good approximation of $u_R$. Nevertheless, note that $u_{\rm app}^N$ already naturally satisfies
\begin{equation*}
	\left\{
	\begin{array}{lcl}
		\div u_{\rm app}^N  =0 & \text{in}& \Omega, \\
		\curl u_{\rm app}^N  =0 & \text{in}& \Omega, \\
		u_{\rm app}^N  \rightarrow 0 & \text{as}& x\rightarrow\infty.
	\end{array}
	\right.
\end{equation*}
Therefore, the key idea lies in enforcing that the boundary and circulation conditions be satisfied as $N\rightarrow\infty$ by setting $\gamma^N=(\gamma_{1}^N,\dots, \gamma_{N}^N)\in\mathbb{R}^N$ to be the solution of the following system of $N$ linear equations:
\begin{equation}\label{point vortex}
\begin{aligned}
& \frac1{2\pi}\sum_{j=1}^N \frac{\gamma_{j}^N}N \frac{(\tilde x_{i}^N - x_{j}^N)^\perp}{|\tilde x_{i}^N - x_{j}^N|^2}\cdot n(\tilde x_{i}^N)
= -[u_{P}\cdot n](\tilde x_{i}^N)
, \quad \text{for all }i=1,\dots, N-1,\\
& \sum_{i=1}^N \frac{\gamma_{i}^N}N = \gamma.
\end{aligned}
\end{equation}
It will be shown later on, under suitable hypotheses on the placement of point vortices, that the above system always has a solution $\gamma^N$. The fact that $u_{\rm app}^N$ is a good approximation of $u_R$ is precisely the content of our main theorem below (see Theorem \ref{main theo}). Clearly, it will then follow that $u$ is well approximated by $u_{\rm app}^N+K_{\mathbb{R}^2}[\omega]$, which will conclude the rigorous justification of the vortex method for the boundary of the exterior of a disk. Other more complicated non-smooth exterior domains are investigated in the full version of our work \cite{ADL}.

Notice that it is now also possible to combine the vortex method for the boundary of an exterior domain with the aforementioned vortex method in the whole plane in order to obtain a full and dynamic vortex method for an exterior domain. To this end, we consider an approximation of the initial vorticity $\omega_{0}$ by a combination of point vortices $\sum_{k=1}^{M} \alpha_{k} \delta_{y_{k}(0)}$. Then, the position $y_{k}(t)$ of each point vortex is let evolve under the influence of the vector field created by the remaining vortices $\sum_{p\neq k} \alpha_{p} \delta_{y_{p}(t)}$ (with possible regularization of the kernel) and the fixed vortices on the boundary $\sum_{i=1}^{N} \frac{\gamma_{i}^N(t)}N \delta_{x_{i}^N}$, where the variable vortex density $\gamma_{i}^N(t)$ is computed through \eqref{point vortex} where $u_{P}$ is replaced by $=K_{\R^2} [ \sum_{k=1}^{M} \alpha_{k} \delta_{y_{k}(t)} ]$.

Finally, it is to be emphasized that the main novelty of this method, when compared to the standard point vortex and vortex-blob methods, is the computation of $\gamma^{N}$ though \eqref{point vortex} allowing the construction of an approximate flow $K_{\R^2}[\omega]+ u_{\rm app}^N$ which only requires the use of the Biot-Savart kernel in the whole plane and does not resort to \eqref{BS exterior}.

\subsection{Main result}

For simplicity, we only consider in this work the stationary case where the points $\{x_{i}^N\}_{i=1}^N$ and $\{\tilde x_{i}^N\}_{i=1}^{N}$ are uniformly distributed on the unit circle:
\begin{equation}\label{mesh}
\theta_{i}^N = \frac{(i-1) 2\pi}{N} \quad \text{and} \quad \tilde\theta_{i}^N = \frac{(i-\frac 12) 2\pi}{N} \quad \forall i=1,\dots, N.
\end{equation}

Our main result states that the approximate flow $u_{\rm app}^N$, constructed through the procedure \eqref{point vortex}, is a good approximation of $u_{R}$:

\begin{theorem}\label{main theo}
	Let $\omega\in C^{0,\alpha}_c\left(\Omega\right)$ (with $0<\alpha\leq 1$) and $\gamma\in\mathbb{R}$ be given. For any $N\geq 2$, we consider the uniformly distributed mesh \eqref{mesh} and $u_{P}$ defined in \eqref{uP}.
	
	Then, the system \eqref{point vortex} admits a unique solution $\gamma^N\in \R^N$. Moreover, for any closed set $K\subset \Omega$  there exists a constant $C=C(K)$ independent of $N$ such that
	\begin{equation*}
		\| u_{R} - u_{\rm app}^N \|_{L^\infty(K)} \leq \frac{C}{N^2},
	\end{equation*}
	where $u_{\rm app}^N$ is given by \eqref{approx} in terms of $\gamma^N$ and $u_{R}$ is the continuous flow \eqref{uR}.
\end{theorem}

Notice that in this particular case of the unit disk, we also have an explicit formula for $u_R$ thanks to \eqref{BS exterior}:
\begin{equation}\label{uR form}
	\begin{aligned}
		u_{R}(x) & = -\frac 1{2\pi} \int_{\Omega} \frac{(x-y^*)^\perp}{|x-y^*|^2} \omega(y)\, dy + \frac{\alpha}{2\pi} \frac{x^\perp}{|x|^2}
		\\
		& = -\frac 1{2\pi} \int_{B(0,1)} \frac{(x-y)^\perp}{|x-y|^2} \frac{1}{|y|^4}\omega(y^*)\, dy + \frac{\alpha}{2\pi} \frac{x^\perp}{|x|^2}.
	\end{aligned}
\end{equation}
Numerically, we indeed verify that the system \eqref{point vortex} is always invertible, and that the $L^\infty$-norm, on any compact set $K$, of the difference of $u_{\rm app}^N$ (given in \eqref{approx}) with $u_{R}$ (given in \eqref{uR form}) decreases as $1/N^2$, which is exactly the rate obtained in Theorem \ref{main theo}. This rate is therefore optimal, at least from the numerical viewpoint. It would be interesting to obtain a rigorous proof of optimality.

The remainder of this article is composed of four parts. In the following section, we establish important representation formulas for the solution of \eqref{eq uR}, which will be used in the proof of our main theorem, and we show the link between our problem and the circular Hilbert transform. Then, in Section \ref{section hilbert}, we prove that the linear system \eqref{point vortex} is invertible. In Section \ref{sect:conv}, we establish that $(u_{R}-u_{\rm app}^N)\cdot n\vert_{\partial \Omega}$ converges to zero in a weak sense. Finally, in the last section, we deduce that such a weak convergence implies the conclusion of Theorem \ref{main theo}.

Thus, the goal of this article is to give a first and simpler poof of validity of the vortex method when restricted to the particular case of the disk where the points $\{x_{i}^N,\tilde x_{i}^N\}_{i=1}^N$ are uniformly distributed. The full general case of an arbitrary exterior domain and more generally distributed meshes is treated in \cite{ADL}. Therein, we also consider the time dependence of the flow and non-zero velocities at infinity.

\begin{remark}
Removing the harmonic part $x^\perp/|x|^2$ and the circulation condition in \eqref{elliptic2} and \eqref{eq uR}, the main result can be readily adapted to describe an ideal fluid inside the unit disk (see \cite{ADL} for more details).
\end{remark}

\section{Boundary vortex sheets and the circular Hilbert transform}

We present now two distinct representation formulas --~other than \eqref{uR form}~-- for the solution $u_R$ of \eqref{eq uR}, which will be crucial for the justification of Theorem \ref{main theo} and whose understanding will shed light on the approximation of $u_R$ by point vortices on the boundary $\partial \Omega$.

Recall that we are considering some given vorticity $\omega\in C^{0,\alpha}_c\left(\Omega\right)$, with $0<\alpha\leq 1$, and $\gamma\in\mathbb{R}$, and wish to construct a velocity field $u_R\in C^0\left(\overline\Omega\right)\cap C^1\left(\Omega\right)$ solving \eqref{eq uR}.

We show now that it is possible to express the solution to \eqref{eq uR} as a vortex sheet on the boundary $\partial\Omega$, which, again, is consistent with the physical idea that the flow around an obstacle is produced by a boundary layer of vortices.

More precisely, we claim that $u_R$ can be expressed as a boundary vortex sheet:
\begin{equation}\label{boundary sheet}
	\begin{aligned}
		v(x) = & K_{\mathbb{R}^2}\left[g\delta_{\partial\Omega}\right]
		=\frac 1{2\pi}\int_{\partial\Omega} \frac{(x-y)^\perp}{|x-y|^2}g(y)dy \\
		= & -\frac 1{2\pi}\int_{\partial\Omega} \frac{(x-y)\cdot \tau(y)}{|x-y|^2}n(y)g(y)dy \\
		& + \frac 1{2\pi}\int_{\partial\Omega} \frac{(x-y)\cdot n(y)}{|x-y|^2}\tau(y)g(y)dy
		\in C^\infty\left(\mathbb{R}^2\setminus\partial\Omega\right),
	\end{aligned}
\end{equation}
for some suitable $g\in C^{0,\alpha}\left(\partial\Omega\right)$, with $0<\alpha\leq 1$. Notice that \eqref{approx} is essentially a discretization of \eqref{boundary sheet}.

Indeed, the theory of single and double layer potentials (or of Cauchy integrals, see \cite{musk}) instructs us that, for a smooth boundary $\partial\Omega$ and for any $g\in C^{0,\alpha}\left(\partial\Omega\right)$, the flow defined by \eqref{boundary sheet} is continuous up to the boundary $\partial\Omega$ (see \cite[Chap. 2, \S\ 16]{musk}), that is $v\in C\left(\overline \Omega\right)\cup C\left(\Omega^c\right)$, and that the limiting values of $v$ on $\partial\Omega$ are given by (see \cite[Chap. 2, \S\ 17]{musk})\begin{equation*}
	\begin{aligned}
		& \lim_{\substack{x\rightarrow x_0\in\partial\Omega \\ x\in\Omega\cup\overline{\Omega}^c}}\frac 1{2\pi}\int_{\partial\Omega} \frac{(x-y)\cdot \tau(y)}{|x-y|^2}n(y)g(y)dy
		=
		\frac 1{2\pi}\int_{\partial\Omega} \frac{(x_0-y)\cdot \tau(y)}{|x_0-y|^2}n(y)g(y)dy,
		\\
		& \lim_{\substack{x\rightarrow x_0\in\partial\Omega \\ x\in\Omega}}\frac 1{2\pi}\int_{\partial\Omega} \frac{(x-y)\cdot n(y)}{|x-y|^2}\tau(y)g(y)dy
		=
		\frac 1{2\pi}\int_{\partial\Omega} \frac{(x_0-y)\cdot n(y)}{|x_0-y|^2}\tau(y)g(y)dy
		+\frac 12 \tau(x_0)g(x_0),
		\\
		& \lim_{\substack{x\rightarrow x_0\in\partial\Omega \\ x\in\overline{\Omega}^c}}\frac 1{2\pi}\int_{\partial\Omega} \frac{(x-y)\cdot n(y)}{|x-y|^2}\tau(y)g(y)dy
		=
		\frac 1{2\pi}\int_{\partial\Omega} \frac{(x_0-y)\cdot n(y)}{|x_0-y|^2}\tau(y)g(y)dy
		-\frac 12 \tau(x_0)g(x_0),
	\end{aligned}
\end{equation*}
where the integral in the right-hand side of the first equation above is defined in the sense of Cauchy's principal value (note that, in the remaining equations, all integrals are defined in the usual sense).

Hence, we deduce that
\begin{equation*}
	\lim_{\substack{x\rightarrow x_0\in\partial\Omega \\ x\in\Omega}} v(x)=
	\frac 1{2\pi}\int_{\partial\Omega} \frac{(x_0-y)^\perp}{|x_0-y|^2}g(y)dy +
	\frac 12 \tau(x_0)g(x_0),
\end{equation*}
and
\begin{equation*}
	\lim_{\substack{x\rightarrow x_0\in\partial\Omega \\ x\in\overline{\Omega}^c}} v(x)=
	\frac 1{2\pi}\int_{\partial\Omega} \frac{(x_0-y)^\perp}{|x_0-y|^2}g(y)dy -
	\frac 12 \tau(x_0)g(x_0),
\end{equation*}
where, again, the integrals in the right-hand sides above are defined in the sense of Cauchy's principal value.

Therefore, we conclude that the flow $v(x)$ given by \eqref{boundary sheet} defines the unique solution $u_R(x)\in C^0\left(\overline\Omega\right)\cap C^1\left(\Omega\right)$ of \eqref{eq uR} if and only if $g\in C^{0,\alpha}\left(\partial\Omega\right)$ satisfies
\begin{equation}\label{density 1}
	\frac 1{2\pi}\int_{\partial\Omega} \frac{(x-y)^\perp}{|x-y|^2}\cdot n(x)g(y)dy
	=u_R\cdot n(x)=-u_P\cdot n(x),
	\quad\text{for every }x\in\partial\Omega,
\end{equation}
and
\begin{equation}\label{density 2}
	\begin{aligned}
		\int_{\partial\Omega} g(x) dx
		= &
		\int_{\partial\Omega}\left(\frac 1{2\pi}\int_{\partial\Omega} \frac{(x-y)^\perp}{|x-y|^2}g(y)dy +
		\frac 12 \tau(x)g(x)\right)\cdot\tau(x)dx
		\\
		& -\int_{\partial\Omega}\left(\frac 1{2\pi}\int_{\partial\Omega} \frac{(x-y)^\perp}{|x-y|^2}g(y)dy -
		\frac 12 \tau(x)g(x)\right)\cdot\tau(x)dx
		\\
		= & \int_{\partial\Omega} u_R\cdot\tau(x)dx - \int_{\overline{\Omega}^c}\curl v(x)dx = \gamma.
	\end{aligned}
\end{equation}

Again, we insist on the fact that the representation formula \eqref{boundary sheet} for the solution of system \eqref{eq uR} only involves the usual Biot-Savart kernel in the whole plane.

The existence of such a density $g\in C^{0,\alpha}\left(\partial\Omega\right)$ satisfying conditions \eqref{density 1} and \eqref{density 2} for any suitable given data is nontrivial, which we address now in the case of the unit disk only.

To this end, note that the singularity of the Biot-Savart kernel satisfies, for all $x,y\in\partial B(0,1)$, that
\begin{equation}\label{cot}
		\frac{\left(x-y\right)^\perp}{\left|x-y\right|^2}\cdot n(x)
		= \frac{-y^\perp\cdot x}{\left|x-y\right|^2}
		= \frac{-\cos\left(\frac\pi 2 + \phi - \theta\right)}{4\sin^2\left(\frac{\phi-\theta}{2}\right)}
		= \frac{\sin\left(\phi - \theta\right)}{4\sin^2\left(\frac{\phi-\theta}{2}\right)}
		= -\frac{1}{2}\cot\left(\frac{\theta-\phi}{2}\right),
\end{equation}
where $x=(\cos\theta,\sin\theta)$ and $y=(\cos\phi,\sin\phi)$, which is nothing but the kernel of the circular Hilbert transform.

Therefore, system \eqref{density 1}-\eqref{density 2} can be recast as
\begin{equation}\label{hilbert}
	\begin{aligned}
		& \int_0^{2\pi} \cot\left(\frac{\theta-\phi}{2}\right)g(\phi)d\phi = f(\theta), \quad \text{for every }\theta\in[0,2\pi],
		\\
		& \int_0^{2\pi}g(\phi)d\phi = \gamma,
	\end{aligned}
\end{equation}
where the $2\pi$-periodic function $g\in C^{0,\alpha}([0,2\pi])$, for some $0<\alpha\leq1$, is the unknown and the $2\pi$-periodic function $f:\R\rightarrow\R$ is defined by
\begin{equation}\label{f}
	f(\theta) = 4\pi[u_P\cdot n](\cos\theta,\sin\theta)\in C^\infty\left([0,2\pi]\right).
\end{equation}
As $u_P$ is smooth and divergence free, we note by the Stokes formula that
\begin{equation}\label{flux}
	\int_{0}^{2\pi} f=0.
\end{equation}

Clearly, solving system \eqref{hilbert} amounts to inverting the circular Hilbert transform
\begin{equation*}
	Hg(\theta) = \int_0^{2\pi} \cot\left(\frac{\theta-\phi}{2}\right)g(\phi)d\phi
	=-i\sum_{k\in\mathbb{Z}}\mathrm{sign}(k)\hat g(k)e^{ik\theta},
\end{equation*}
which is a well-known involution on the space of zero-mean periodic functions in $L^2([0,2\pi])$, that is to say
\begin{equation}\label{involution}
	H^2g(\theta)=-4\pi^2\left(g(\theta)-\frac 1{2\pi}\int_0^{2\pi}g(\phi)d\phi\right),\quad \text{for all }g\in L^2([0,2\pi]).
\end{equation}
It therefore follows that the solution to \eqref{hilbert} (or, equivalently, to \eqref{density 1}-\eqref{density 2}) is given by
\begin{equation*}
	g(\theta) = \frac{-1}{4\pi^2}Hf(\theta) + \frac{\gamma}{2\pi}
	= \frac{-1}{\pi}H[u_P\cdot n](\theta) + \frac{\gamma}{2\pi}\in C^\infty\left([0,2\pi]\right),
\end{equation*}
and is smooth, for $f$ is smooth, whereby, in view of \eqref{boundary sheet}, we obtain the following representation formula on the exterior of a disk:
\begin{equation*}
		u_R(x)
		=-\frac 1{2\pi^2}\int_{\partial B(0,1)} \frac{(x-y)^\perp}{|x-y|^2}H[u_P\cdot n](y)dy
		+\frac \gamma{4\pi^2}\int_{\partial B(0,1)} \frac{(x-y)^\perp}{|x-y|^2}dy.
\end{equation*}
In particular, we deduce, by comparing the above identity with \eqref{uR form} and by uniqueness of solutions to system \eqref{eq uR}, that it holds
\begin{equation}\label{ball circulation}
	\frac 1{2\pi}\int_{\partial B(0,1)} \frac{(x-y)^\perp}{|x-y|^2}dy = \frac{x^\perp}{|x|^2},\quad\text{for every }x\in\Omega.
\end{equation}
Thus, we finally conclude that
\begin{equation}\label{boundary sheet circle}
		u_R(x)
		=-\frac 1{2\pi^2}\int_{\partial B(0,1)} \frac{(x-y)^\perp}{|x-y|^2}H[u_P\cdot n](y)dy
		+\frac \gamma{2\pi} \frac{x^\perp}{|x|^2}.
\end{equation}

It turns out that there is yet another convenient representation formula for the flow $u_R$, which is a variant of the boundary vortex sheet \eqref{boundary sheet}.

More precisely, we claim now that in the exterior of a disk, $u_R$ can also be expressed as:
\begin{equation}\label{boundary sheet 2}
	\begin{aligned}
		w(x) = & \frac 1{2\pi}\int_{\partial\Omega} \frac{x-y}{|x-y|^2}h(y)dy
		+ \frac \gamma{2\pi} \frac{x^\perp}{|x|^2}\\
		= & \frac 1{2\pi}\int_{\partial\Omega} \frac{(x-y)\cdot n(y)}{|x-y|^2}n(y)h(y)\, dy + \frac 1{2\pi}\int_{\partial\Omega} \frac{(x-y)\cdot \tau(y)}{|x-y|^2}\tau(y)h(y)\,dy \\
		& + \frac \gamma{2\pi} \frac{x^\perp}{|x|^2}
		\in C^\infty\left(\mathbb{R}^2\setminus\partial\Omega\right),
	\end{aligned}
\end{equation}
for some suitable $h\in C^{0,\alpha}\left(\partial\Omega\right)$, with $0<\alpha\leq 1$.

As before, the theory of single and double layer potentials instructs us that, for a smooth boundary $\partial\Omega$ and for any $h\in C^{0,\alpha}\left(\partial\Omega\right)$, the flow $w$ is continuous up to the boundary $\partial\Omega$, that is $w\in C\left(\overline \Omega\right)\cup C\left(\Omega^c\right)$, and that the limiting values of $w$ on $\partial\Omega$ are given by
\begin{equation*}
	\begin{aligned}
		 &\lim_{\substack{x\rightarrow x_0\in\partial\Omega \\ x\in\Omega\cup\overline{\Omega}^c}}  \frac 1{2\pi}\int_{\partial\Omega} \frac{(x-y)\cdot \tau(y)}{|x-y|^2}
		\tau(y)h(y)\, dy
		 =
		\frac 1{2\pi}\int_{\partial\Omega} \frac{(x_0-y)\cdot \tau(y)}{|x_0-y|^2}\tau(y)h(y)\, dy,
		\\
		 &\lim_{\substack{x\rightarrow x_0\in\partial\Omega \\ x\in\Omega}}\frac 1{2\pi}\int_{\partial\Omega} \frac{(x-y)\cdot n(y)}{|x-y|^2}n(y)h(y)\,dy
		 =
		\frac 1{2\pi}\int_{\partial\Omega} \frac{(x_0-y)\cdot n(y)}{|x_0-y|^2}n(y)h(y)\, dy
		+\frac 12 n(x_0)h(x_0),
		\\
		& \lim_{\substack{x\rightarrow x_0\in\partial\Omega \\ x\in\overline{\Omega}^c}}\frac 1{2\pi}\int_{\partial\Omega} \frac{(x-y)\cdot n(y)}{|x-y|^2}n(y)h(y)\, dy
		 =
		\frac 1{2\pi}\int_{\partial\Omega} \frac{(x_0-y)\cdot n(y)}{|x_0-y|^2}n(y)h(y)\, dy
		-\frac 12 n(x_0)h(x_0),
	\end{aligned}
\end{equation*}
where the integral in the right-hand side of the first equation above is defined in the sense of Cauchy's principal value (note that, in the remaining equations, all integrals are defined in the usual sense).

Hence, we deduce that
\begin{equation*}
	\lim_{\substack{x\rightarrow x_0\in\partial\Omega \\ x\in\Omega}} w(x)=
	\frac 1{2\pi}\int_{\partial\Omega} \frac{x_0-y}{|x_0-y|^2}h(y)dy +
	\frac 12 n(x_0)h(x_0)+ \frac \gamma{2\pi} \frac{x_{0}^\perp}{|x_{0}|^2},
\end{equation*}
and
\begin{equation*}
	\lim_{\substack{x\rightarrow x_0\in\partial\Omega \\ x\in\overline{\Omega}^c}} w(x)=
	\frac 1{2\pi}\int_{\partial\Omega} \frac{x_0-y}{|x_0-y|^2}h(y)dy -
	\frac 12 n(x_0)h(x_0) + \frac \gamma{2\pi} \frac{x_{0}^\perp}{|x_{0}|^2},
\end{equation*}
where, again, the integrals in the right-hand sides above are defined in the sense of Cauchy's principal value.

Therefore, we conclude that the flow $w(x)$ given by \eqref{boundary sheet 2} defines the unique solution $u_R(x)\in C^0\left(\overline\Omega\right)\cap C^1\left(\Omega\right)$ of \eqref{eq uR} if and only if $h\in C^{0,\alpha}\left(\partial\Omega\right)$ satisfies
\begin{equation}\label{density 3}
	\frac 1{2\pi}\int_{\partial\Omega} \frac{x-y}{|x-y|^2}\cdot n(x)h(y)dy + \frac 12 h(x)
	=u_R\cdot n(x)=-u_P\cdot n(x),
	\quad\text{for every }x\in\partial\Omega,
\end{equation}
and
\begin{equation}\label{density 4}
	\begin{aligned}
		\int_{\partial\Omega} h(x) dx
		= &
		\int_{\partial\Omega}\left(\frac 1{2\pi}\int_{\partial\Omega} \frac{x-y}{|x-y|^2}h(y)dy +
		\frac 12 n(x)h(x) + \frac{\gamma}{2\pi}\frac{x^\perp}{|x|^2}\right)\cdot n(x)dx
		\\
		& -\int_{\partial\Omega}\left(\frac 1{2\pi}\int_{\partial\Omega} \frac{x-y}{|x-y|^2}h(y)dy -
		\frac 12 n(x)h(x)+ \frac{\gamma}{2\pi}\frac{x^\perp}{|x|^2}\right)\cdot n(x)dx
		\\
		= & \int_{\partial\Omega} u_R\cdot n(x)dx - \int_{\overline{\Omega}^c}\div w(x)dx = -\int_{\partial\Omega} u_P\cdot n(x)dx=0.
	\end{aligned}
\end{equation}
Note that the circulation condition
\begin{equation*}
	\begin{aligned}
		\int_{\partial\Omega} u_R\cdot\tau(x)dx
		= &
		\int_{\partial\Omega}\left(\frac 1{2\pi}\int_{\partial\Omega} \frac{x-y}{|x-y|^2}h(y)dy +
		\frac 12 n(x)h(x) + \frac \gamma{2\pi} \frac{x^\perp}{|x|^2} \right)\cdot\tau(x)dx
		\\
		= & \int_{\partial\Omega}\left(\frac 1{2\pi}\int_{\partial\Omega} \frac{x-y}{|x-y|^2}h(y)dy -
		\frac 12 n(x)h(x)\right)\cdot\tau(x)dx + \gamma
		\\
		= & \int_{\overline{\Omega}^c}\curl \left(\frac 1{2\pi}\int_{\partial\Omega} \frac{x-y}{|x-y|^2}h(y)dy\right)   dx + \gamma = \gamma,
	\end{aligned}
\end{equation*}
is automatically satisfied.

The existence of such a density $h\in C^{0,\alpha}\left(\partial\Omega\right)$ satisfying conditions \eqref{density 3} and \eqref{density 4} for any suitable given data is nontrivial, which we address now in the case of the unit disk only.

To this end, note that the singularity of the Biot-Savart kernel satisfies, for all $x,y\in\partial B(0,1)$, that
\begin{equation*}
		\frac{x-y}{\left|x-y\right|^2}\cdot n(x)
		= \frac{1-y\cdot x}{\left|x-y\right|^2}
		= \frac{1-\cos\left(\phi - \theta\right)}{4\sin^2\left(\frac{\phi-\theta}{2}\right)}
		= \frac{1}{2},
\end{equation*}
where $x=(\cos\theta,\sin\theta)$ and $y=(\cos\phi,\sin\phi)$.

Therefore, it is readily seen that system \eqref{density 3}-\eqref{density 4} is uniquely solved by
\begin{equation*}
	h(\theta)=-2u_P\cdot n(\theta)\in C^\infty\left([0,2\pi]\right),
\end{equation*}
whereby, in view of \eqref{boundary sheet 2}, we obtain the following representation formula on the exterior of a disk:
\begin{equation}\label{boundary sheet circle 2}
	u_R(x) = -\frac 1{\pi}\int_{\partial B(0,1)} \frac{x-y}{|x-y|^2}(u_P\cdot n)(y)dy
	+ \frac \gamma{2\pi} \frac{x^\perp}{|x|^2}.
\end{equation}

It then follows, by comparing \eqref{boundary sheet circle 2} with \eqref{boundary sheet circle} and by uniqueness of solutions to system \eqref{eq uR}, that
\begin{equation*}
	\int_{\partial B(0,1)} \frac{x-y}{|x-y|^2}(u_P\cdot n)(y)dy
	=
	\frac 1{2\pi}\int_{\partial B(0,1)} \frac{(x-y)^\perp}{|x-y|^2}H[u_P\cdot n](y)dy,
	\quad\text{for every }x\in\Omega,
\end{equation*}
whence we infer that, replacing $u_P\cdot n$ by $H\varphi$ in view of the arbitrariness of zero-mean boundary data in \eqref{eq uR} and using the inversion of the Hilbert transform \eqref{involution},
\begin{equation*}
	\begin{aligned}
		\int_{\partial B(0,1)} \frac{x-y}{|x-y|^2}H\varphi(y)dy
		& =
		\frac 1{2\pi}\int_{\partial B(0,1)} \frac{(x-y)^\perp}{|x-y|^2}H^2\varphi(y)dy
		\\ & =
		-{2\pi}\int_{\partial B(0,1)} \frac{(x-y)^\perp}{|x-y|^2}\left(\varphi(y)-\frac1{2\pi}\int_{\partial B(0,1)}\varphi(z)dz\right)dy,
		\\
		& \text{for every }x\in\Omega.
	\end{aligned}
\end{equation*}
Hence, for $y=(\cos\phi,\sin\phi)\in\partial B(0,1)$ and $x\in \Omega$, we have by \eqref{ball circulation}:
	\begin{align*}
		\int_{0}^{2\pi}\int_{0}^{2\pi} \frac{x-y}{|x-y|^2} \cot\left(\frac{\phi-\theta}2 \right) \varphi(\theta)d\theta d\phi
		& =
		-{2\pi}\int_{0}^{2\pi} \frac{(x-y)^\perp}{|x-y|^2}\varphi(\phi) d\phi +2\pi   \frac{x^\perp}{|x|^2}    \int_{\partial B(0,1)}\varphi(z)dz,\\
		\int_{0}^{2\pi}\int_{0}^{2\pi} \frac{x-z}{|x-z|^2} \cot\left(\frac{\theta-\phi}2 \right) \varphi(\phi)d\phi d\theta
		& =
		-{2\pi}\int_{0}^{2\pi} \frac{(x-y)^\perp}{|x-y|^2}\varphi(\phi) d\phi +2\pi   \frac{x^\perp}{|x|^2}   \int_{0}^{2\pi}\varphi(\phi) d\phi,
	\end{align*}
where $z=(\cos\theta,\sin\theta)$.
Finally, by the arbitrariness of $\varphi$, we conclude
\begin{equation}\label{vortex identity}
	\begin{aligned}
		\frac{1}{2\pi}\int_{\partial B(0,1)}
		\frac{x-z}{|x-z|^2}
		\cot\left(\frac{\phi-\theta}{2}\right) dz
		=
		\frac{(x-y)^\perp}{|x-y|^2}
		-\frac{x^\perp}{|x|^2},
		\quad\text{for every }x\in\Omega\text{ and }y\in\partial B(0,1),
	\end{aligned}
\end{equation}
which will be useful later on. Once again, we insist on the fact that the above integral is defined is the sense of Cauchy's principal value.

\begin{remark}
As explained in the introduction, our goal is to justify that $u_{\rm app}^N$ \eqref{approx} is a good discretization of the formulation \eqref{boundary sheet circle}. In fact, it would be easier, at least numerically, to discretize  \eqref{boundary sheet circle 2} which would give us a direct approximation of $u_{R}$ without inverting large matrices (related to the computation of the inverse Hilbert transform $H^{-1}$; see Section \ref{section hilbert}).

For more general geometries of $\Omega$, we prove in \cite{ADL} that there also exist densities $g$ and $h$ satisfying the above conditions \eqref{density 1}, \eqref{density 2}, \eqref{density 3} and \eqref{density 4}. However, \eqref{boundary sheet circle 2} does not hold anymore and so, the processes to get $g$ or $h$ involve similar difficulties.
\end{remark}

\section{Solving system \eqref{point vortex} and the discrete circular Hilbert transform}\label{section hilbert}

Using \eqref{cot} and considering the angles $\{ \theta_i^N \}$ and $\{ \tilde \theta_i^N \} $ associated to $\{ x_i^N \}$ and $\{ \tilde x_i^N \} $ (see \eqref{xi}-\eqref{tildexi}), the system \eqref{point vortex} of $N$ equations can be recast as
\begin{equation}\label{point toy}
	\begin{aligned}
	&	\frac1N\sum_{j=1}^N \gamma_{j}^N \cot\left(\frac{\tilde\theta_{i}^N-\theta_{j}^N}2\right)  = f(\tilde \theta_{i}^N), \quad \text{for all }i=1,\dots, N-1,\\
	&	\frac1N\sum_{i=1}^N \gamma_{i}^N  = \gamma,
	\end{aligned}
\end{equation}
where $\gamma^N=(\gamma_{1}^N,\dots, \gamma_{N}^N)\in\mathbb{R}^N$ is the unknown and $f$ is defined in \eqref{f}. Loosely speaking, solving system \eqref{point toy} amounts to inverting a discrete Hilbert transform on the circle. Indeed, \eqref{point toy} clearly is a discretization of \eqref{hilbert}.

From now on, we will also conveniently denote the matrices:
\[
A_{N-1,N} := 
\left(\cot\left(\frac{\tilde\theta_{i}^N-\theta_{j}^N}2\right)
\right)_{1\leq i\leq N-1,1\leq j\leq N}
\quad \text{and}\quad 
A_{N} := 
\left(\cot\left(\frac{\tilde\theta_{i}^N-\theta_{j}^N}2\right)
\right)_{1\leq i, j\leq N},
\]
and we will make use of the following notations for $z\in \R^N$:
\begin{equation*}
	\begin{aligned}
		\| z\|_{\ell^p} & := \Big(\frac1N \sum_{i=1}^N |z_{i}|^p \Big)^{1/p}, \quad \text{for any }p\in [1,\infty),\\
		\|z\|_{\ell^\infty} & := \max_{i=1,\dots,N}  |z_{i}|,\\
		\langle z \rangle & := \frac1N\sum_{i=1}^N z_{i}.
	\end{aligned}
\end{equation*}
Note that, with this normalization of the norms, we have:
\[
\|z\|_{\ell^p}\leq \| z \|_{\ell^q}, \text{ for any } 1\leq p\leq q \leq \infty.
\]

Finally, for the uniformly distributed mesh \eqref{mesh}, notice that, by odd symmetry of the cotangent function,
\begin{equation}\label{perfect distri}
	\sum_{1\leq j \leq N} \cot\left(\frac{\tilde\theta_{i}^N-\theta_{j}^N}2\right)=0,
	\qquad\text{and}\qquad
	\sum_{1\leq j \leq N} \cot\left(\frac{\tilde\theta_{j}^N-\theta_{i}^N}2\right) =0,
\end{equation}
for each $i=1,\ldots,N$. In fact, it can be shown (see \cite{ADL}) that the only possible mesh satisfying \eqref{perfect distri} and $\theta_1^N=0$ is necessarily given by \eqref{mesh}.

\begin{remark}
These cancellations will be used several times in the following proofs and are related with the continuous version $\displaystyle\int_{0}^{2\pi} \cot\left(\frac{\phi-\theta}2\right) d\theta =0$. As the oddness of the cotangent function plays a crucial role to define the Cauchy's principal value, the symmetry of the points $(\theta_{i}^N,\tilde \theta_{i}^N)$ is important to get a suitable discretization of the Hilbert transform.
\end{remark}

The first result in this section is a precise $\ell^2$-estimate on $A_N$ for the uniformly distributed mesh \eqref{mesh}.
\begin{proposition}\label{est l2}
	Consider the uniformly distributed mesh $(\theta_{1}^N,\dots , \theta_{N}^N),(\tilde \theta_{1}^N, \dots , \tilde \theta_{N}^N)\in [0,2\pi)^N$ defined by \eqref{mesh}.
	
	Then, for any $z\in\mathbb{R}^N$, we have that
	\[
	\| z - \langle z \rangle \mathbf{1} \|_{\ell^2}= \frac1N \| A_{N} z \|_{\ell^2},
	\]
	where $\mathbf{1}=(1,\ldots,1)\in\mathbb{R}^N$.
\end{proposition}

\begin{proof}
	First, we compute
	\begin{equation*}
	\begin{split}
	N\| A_{N} z \|_{\ell^2}^2 
	=& \sum_{1\leq k \leq N} \Big| \sum_{1\leq j \leq N} \cot\left(\frac{\tilde\theta_{k}^N-\theta_{j}^N}2\right)z_{j}\Big|^2 =  \sum_{1\leq k \leq N}  \sum_{1\leq i,j \leq N} \cot\left(\frac{\tilde\theta_{k}^N-\theta_{i}^N}2\right)\cot\left(\frac{\tilde\theta_{k}^N-\theta_{j}^N}2\right)z_{i} z_{j} \\
	=&-\frac12 \sum_{1\leq i,j \leq N} (z_{i} - z_{j})^2 \sum_{1\leq k \leq N}  \cot\left(\frac{\tilde\theta_{k}^N-\theta_{i}^N}2\right)\cot\left(\frac{\tilde\theta_{k}^N-\theta_{j}^N}2\right)\\
	&+ \sum_{1\leq i,k \leq N} |z_{i}|^2   \cot\left(\frac{\tilde\theta_{k}^N-\theta_{i}^N}2\right)  \sum_{1\leq j \leq N} \cot\left(\frac{\tilde\theta_{k}^N-\theta_{j}^N}2\right).
	\end{split}
	\end{equation*}
	Note that the last sum in the right-hand side is equal to zero by \eqref{perfect distri}.
	
	As for the remaining term above, we use the following elementary relation, valid for any $a,b$ such that $a,b,a-b \notin \pi \Z$:
	\[
		\cot a \cot b = \cot (b-a)[\cot a - \cot b]-1,
	\]
	to write
	\begin{equation*}
	\begin{split}
	N\| A_{N} z \|_{\ell^2}^2 
	=& -\frac12 \sum_{1\leq i\neq j \leq N} (z_{i} - z_{j})^2  \cot\left(\frac{\theta_{i}^N-\theta_{j}^N}2\right) \sum_{1\leq k \leq N} \Bigl[ \cot\left(\frac{\tilde\theta_{k}^N-\theta_{i}^N}2\right)-\cot\left(\frac{\tilde\theta_{k}^N-\theta_{j}^N}2\right) \Bigl]\\
	&+\frac{N}2\sum_{1\leq i\neq j \leq N} (z_{i} - z_{j})^2\\
	=&\frac{N}2\sum_{1\leq i,j \leq N} (z_{i} - z_{j})^2,
	\end{split}
	\end{equation*}
	where we have also used \eqref{perfect distri}.
	
	Finally, the last sum is easily recast as
	\begin{equation*}\begin{split}
	\frac{N}2\sum_{1\leq i,j \leq N} (z_{i} - z_{j})^2& = N\sum_{1\leq i,j \leq N} (z_{i} -\langle z\rangle)^2 - N \sum_{1\leq i,j \leq N} (z_{i} -\langle z\rangle)(z_{j} -\langle z\rangle)\\
	&=N^2 \sum_{1\leq i \leq N} (z_{i} -\langle z\rangle)^2 = N^3 \| z - \langle z\rangle\mathbf{1} \|_{\ell^2}^2.
	\end{split}\end{equation*}
	We have therefore obtained that
	\[
	N\| A_{N} z \|_{\ell^2}^2  = N^3 \| z - \langle z\rangle\mathbf{1} \|_{\ell^2}^2,
	\]
	which ends the proof of the proposition.
\end{proof}

The preceding proposition allows us to get the existence and the uniqueness of the solution to \eqref{point toy}.

\begin{corollary}\label{inverse perfect}
	Consider the uniformly distributed mesh $(\theta_{1}^N,\dots , \theta_{N}^N),(\tilde \theta_{1}^N, \dots , \tilde \theta_{N}^N)\in [0,2\pi)^N$ defined by \eqref{mesh}.
	
	Then, for any $v\in \R^{N-1}$ and $\gamma \in \R$, the following problem:
	\begin{equation}\label{prob}
	z\in \R^N,\quad \frac1N A_{N-1,N}z=v, \quad \langle z \rangle = \gamma,
	\end{equation}
	has a unique solution. Moreover, this solution satisfies:
	\begin{equation}\label{est l1}
	\| z \|_{\ell^1}\leq \| z \|_{\ell^2} \leq \| v \|_{\ell^2} + |\gamma| + \sqrt N \left|\langle v \rangle\right|
	\leq \| v \|_{\ell^\infty} + |\gamma| + \sqrt N \left|\langle v \rangle\right|.
	\end{equation}
\end{corollary}

\begin{proof}
	Let us define
	\begin{equation*}
		\begin{split}
			\Phi\ :\quad \R^N &\to \R^{N+1}\\
			 z &\mapsto  
			\begin{pmatrix}
				\frac1N A_{N} z \\ \langle z \rangle
			\end{pmatrix},
		\end{split}
	\end{equation*}
	which is an injective linear mapping (see Proposition \ref{est l2}). Therefore, $\Phi$ is bijective from $\R^N$ onto ${\rm Im\ }\Phi$.
	
	Moreover, noting that, for any $z\in \R^N$, we have
	\[
	\langle A_{N} z \rangle = \frac1N \sum_{j=1}^N z_{j} \sum_{i=1}^N \cot\left(\frac{\tilde\theta_{i}^N-\theta_{j}^N}2\right)=0,
	\]
	by \eqref{perfect distri}, and that $\dim\left(\mathrm{Im\ }\Phi\right)=N$, we conclude
	\[
	{\rm Im\ }\Phi = \left\{ u \in \R^{N+1}\ :\ \sum_{i=1}^N u_{i} =0 \right\}. 
	\]

	Now, let $v\in \R^{N-1}$ and $\gamma \in \R$ be fixed. There exists a unique $v_{N}$ such that
	$
	\begin{pmatrix}
	v\\v_{N}\\\gamma
	\end{pmatrix}
	\in {\rm Im\ }\Phi
	$, namely $v_{N}=-\sum_{i=1}^{N-1}v_{i}$. With this $v_{N}$, we then deduce the existence of $z\in \R^N$ such that $\Phi(z)=\begin{pmatrix}
	v\\v_{N}\\\gamma
	\end{pmatrix}
	$. In particular $z$ is a solution to \eqref{prob} and, by Proposition \ref{est l2}, it holds that
	\[
	\| z -\gamma\mathbf{1}\|_{\ell^2} = \| z - \langle z\rangle\mathbf{1}\|_{\ell^2} =  \frac1N \| A_{N} z \|_{\ell^2} =\Big( \frac1N \sum_{i=1}^{N}|v_{i}|^2 \Big)^{1/2} = \Big( \frac1N \sum_{i=1}^{N-1}|v_{i}|^2 +\frac1N \Big|\sum_{i=1}^{N-1}v_{i}\Big|^2\Big)^{1/2}.
	\]
	As $\| z \|_{\ell^2}-|\gamma| \leq  \| z -\gamma\mathbf{1}\|_{\ell^2}$ and
	\[
	\Big( \frac1N \sum_{i=1}^{N-1}|v_{i}|^2 +\frac1N \Big|\sum_{i=1}^{N-1}v_{i}\Big|^2\Big)^{1/2}
	\leq\Big( \frac1N \sum_{i=1}^{N-1}|v_{i}|^2\Big)^{1/2} +\frac{1}{\sqrt N} \Big|\sum_{i=1}^{N-1}v_{i}\Big|
	=
	\sqrt{\frac{N-1}N} \|v \|_{\ell^2} +\frac{N-1}{\sqrt N} \left|\langle v \rangle\right|,
	\]
	we conclude that
	\[
	\| z \|_{\ell^2} \leq \| v \|_{\ell^2} + |\gamma| + \sqrt N \left|\langle v \rangle\right|.
	\]
	
	Finally, concerning the uniqueness of a solution to \eqref{prob}, let us consider $z$ and $\tilde z$ two solutions of \eqref{prob}. Then, $\Phi(z-\tilde z)=\begin{pmatrix}
	0_{\R^{N-1}}\\x\\0
	\end{pmatrix}$
	(for some $x\in \R$) belongs to ${\rm Im\ }\Phi$ if only if $x= 0$. By injectivity of $\Phi$, we conclude that necessarily $z=\tilde z$, thereby completing the proof of the corollary.
\end{proof}

\section{Weak convergence of the discrete circular Hilbert transform}\label{sect:conv}

The results in this section will serve to show that $(u_{R}-u_{\rm app}^N)\cdot n\vert_{\partial \Omega}$ vanishes in a weak sense.

The following elementary lemma is a reminder about standard estimates on the rate of convergence of Riemann sums.

\begin{lemma}\label{riemann}
	Consider the uniformly distributed mesh $(\theta_{1}^N,\dots , \theta_{N}^N),(\tilde \theta_{1}^N, \dots , \tilde \theta_{N}^N)\in [0,2\pi)^N$ defined by \eqref{mesh} and let $g$ be a smooth periodic function.
	
	Then, for any $0<\alpha\leq1$ and $k=0,1$,
	\begin{equation*}
		\left|\int_0^{2\pi}g(\theta){\rm d}\theta - \frac{2\pi}{N} \sum_{i=1}^{N} g(\tilde \theta_{i}^N) \right|\leq \frac{C}{N^{k+\alpha}}\|g \|_{C^{k,\alpha}},
	\end{equation*}
	for some independent constant $C>0$.
\end{lemma}

\begin{proof}
	First, a standard estimate yields, setting $\theta_{N+1}^N=2\pi$,
	\begin{equation*}
		\begin{aligned}
			\left|\int_0^{2\pi}g(\theta){\rm d}\theta - \frac{2\pi}{N} \sum_{i=1}^{N} g(\tilde \theta_{i}^N) \right| & =
			\left|\sum_{i=1}^{N}\left(
			\int_{\theta_i^N}^{\theta_{i+1}^N} g(\theta) {\rm d}\theta- \frac{2\pi}{N}  g(\tilde\theta_{i}^N)\right)
			\right|
			\\
			&\leq
			\sum_{i=1}^{N}\left|
			\int_{\theta_i^N}^{\theta_{i+1}^N} \left( g(\theta) - g(\tilde\theta_{i}^N) \right){\rm d}\theta\right|
			\\
			&\leq
			\frac{(2\pi)^{1+\alpha}}{N^{\alpha}}
			\sup_{x,y\in[0,2\pi]}\frac{\left|g(x) - g(y)\right|}{|x-y|^\alpha}
			\leq
			\frac{(2\pi)^{1+\alpha}}{N^{\alpha}} \left\|g\right\|_{C^{0,\alpha}},
		\end{aligned}
	\end{equation*}
	which establishes the lemma when $k=0$.
	
	For the case $k=1$, recalling $\tilde\theta_i^N=\frac{\theta_i^N+\theta_{i+1}^{N}}{2}$, one finds that
	\begin{equation*}
		\begin{aligned}
			\left|\int_0^{2\pi}g(\theta){\rm d}\theta - \frac{2\pi}{N} \sum_{i=1}^{N} g(\tilde \theta_{i}^N) \right|
			&\leq
			\sum_{i=1}^{N}\left|
			\int_{\theta_i^N}^{\theta_{i+1}^N} \left( g(\theta) - g(\tilde\theta_{i}^N) \right){\rm d}\theta\right|
			\\
			&=\frac \pi N
			\sum_{i=1}^{N}\left|
			\int_{0}^1 \left( g\left(\tilde\theta_i^N+\frac\pi N t\right) + g\left(\tilde\theta_i^N-\frac\pi N t\right)
			- 2g(\tilde\theta_{i}^N)\right){\rm d}t\right|
			\\
			&=\frac {\pi^2}{N^2}
			\sum_{i=1}^{N}\left|
			\int_{0}^1\int_0^1 t\left( g'\left(\tilde\theta_i^N+\frac\pi N st\right) - g'\left(\tilde\theta_i^N-\frac\pi N st\right)
			\right){\rm d}t{\rm d}s\right|
			\\
			&\leq
			\frac{\pi^{2+\alpha}}{N^{1+\alpha}}
			\sup_{x,y\in[0,2\pi]}\frac{\left|g'(x) - g'(y)\right|}{|x-y|^\alpha}
			\leq
			\frac{\pi^{2+\alpha}}{N^{1+\alpha}} \left\|g\right\|_{C^{1,\alpha}},
		\end{aligned}
	\end{equation*}
	which concludes the proof of the lemma.
\end{proof}

\begin{proposition}\label{prop 32}
	Consider the uniformly distributed mesh $(\theta_{1}^N,\dots , \theta_{N}^N),(\tilde \theta_{1}^N, \dots , \tilde \theta_{N}^N)\in [0,2\pi)^N$ defined by \eqref{mesh} and, according to Corollary \ref{inverse perfect}, consider the solution $\gamma^N=(\gamma_{1}^N,\dots, \gamma_{N}^N)\in\mathbb{R}^N$ to the system \eqref{point toy} for some periodic function $f\in C^{k,\alpha}\left([0,2\pi]\right)$, where $k=0,1$, $0<\alpha\leq 1$ and $k+\alpha\geq \frac 12$, with zero mean value \eqref{flux} and some $\gamma\in\mathbb{R}$. We define the approximation
	\begin{equation}\label{f app}
		f_{\rm app}^N(\theta):= \frac 1N\sum_{j=1}^N \gamma_{j}^N \cot\left(\frac{\theta-\theta_{j}^N}2\right).
	\end{equation}

	Then, for any periodic test function $\varphi\in C^{k,\alpha}\left([0,2\pi]\right)$,
	\begin{equation*}
		\left|\int_{0}^{2\pi} (f_{\rm app}^N - f )\varphi\right|
		\leq
		\frac{C}{N^{k+\alpha}}
		\left(\left\|f\right\|_{C^{k,\alpha}}+|\gamma|
		\right)
		\left\|\varphi\right\|_{C^{k+1,\alpha}},
	\end{equation*}
	where the singular integrals are defined in the sense of Cauchy's principal value.
\end{proposition}

\begin{proof}
	Let $\varphi \in C^{\infty}\left([0,2\pi]\right)$ be a periodic test function. Then, we decompose
	\begin{equation*}
		\begin{aligned}
			\int_{0}^{2\pi} (f_{\rm app}^N - f )\varphi
			=& \left(
			\int_{0}^{2\pi} f_{\rm app}^N \varphi - \frac{2\pi}{N} \sum_{i=1}^{N} f_{\rm app}^N(\tilde\theta_{i}^N) \varphi(\tilde\theta_{i}^N)
			\right)\\
			&- \left(
			\int_{0}^{2\pi} f \varphi- \frac{2\pi}{N} \sum_{i=1}^{N} f(\tilde\theta_{i}^N) \varphi(\tilde\theta_{i}^N)
			\right)\\
			&+
			\frac{2\pi}{N} \sum_{i=1}^{N-1} \left( f_{\rm app}^N(\tilde\theta_{i}^N)- f(\tilde\theta_{i}^N) \right) \varphi(\tilde\theta_{i}^N)
			\\
			&+
			\frac{2\pi}{N} \left( f_{\rm app}^N(\tilde\theta_{N}^N)- f(\tilde\theta_{N}^N) \right) \varphi(\tilde\theta_{N}^N)
			\\
			=&: D_1 + D_2 + D_3 + D_4.
		\end{aligned}
	\end{equation*}
	It is readily seen that $D_3$ is null, for $f_{\rm app}^N(\tilde\theta_{i}^N) = f(\tilde\theta_{i}^N)$, for all $i=1,\ldots,N-1$, by construction (see \eqref{point toy}).

	Next, note that $D_2$ is the error of approximation of the integral $\int_{0}^{2\pi} f \varphi$ by its Riemann sum. Therefore, a direct application of Lemma \ref{riemann} yields
	\begin{equation}\label{d2}
		|D_2|\leq \frac{C}{N^{k+\alpha}} \left\|f\varphi\right\|_{C^{k,\alpha}}
		\leq\frac{C}{N^{k+\alpha}} \left\|f\right\|_{C^{k,\alpha}}\left\|\varphi\right\|_{C^{k,\alpha}}.
	\end{equation}

	As for the term $D_1$, it is first rewritten, exploiting the symmetry of the cotangent function (see \eqref{perfect distri}), as
	\begin{equation*}
		\begin{aligned}
			D_1  =& \int_{0}^{2\pi} f_{\rm app}^N \varphi - \frac{2\pi}{N} \sum_{i=1}^{N} f_{\rm app}^N(\tilde\theta_{i}^N) \varphi(\tilde\theta_{i}^N)
			\\
			 =& \frac 1N\sum_{j=1}^N \gamma_{j}^N
			\int_{0}^{2\pi} \cot\left(\frac{\theta-\theta_{j}^N}2\right)
			\varphi (\theta)\d \theta
			- \frac{2\pi}{N^2} \sum_{i,j=1}^{N} \gamma_{j}^N
			\cot\left(\frac{\tilde\theta_i^N-\theta_{j}^N}2\right)
			\varphi(\tilde\theta_{i}^N)
			\\
			=& \frac 1N\sum_{j=1}^N \gamma_{j}^N
			\int_{0}^{2\pi} \cot\left(\frac{\theta-\theta_{j}^N}2\right)
			\left(\varphi (\theta) - \varphi(\theta_j^N)\right)\d \theta
			\\
			& - \frac{2\pi}{N^2} \sum_{i,j=1}^{N} \gamma_{j}^N
			\cot\left(\frac{\tilde\theta_i^N-\theta_{j}^N}2\right)
			\left(\varphi(\tilde\theta_{i}^N)-\varphi(\theta_{j}^N)\right)
			\\
			 =&\int_{0}^{2\pi} F(\theta){\rm d} \theta
			- \frac{2\pi}{N} \sum_{i=1}^N
			F(\tilde\theta_i^N),
		\end{aligned}
	\end{equation*}
	where
	\begin{equation*}
		F(\theta)=
		\frac 1N\sum_{j=1}^N \gamma_j^N \cot\left(\frac{\theta-\theta_{j}^N}2\right)
		\left(\varphi (\theta) - \varphi(\theta_j^N)\right).
	\end{equation*}
	Note that the integrand $\theta\mapsto  \cot\left(\frac{\theta-\theta_{j}^N}2\right) (\varphi(\theta)- \varphi(\theta_{j}^N))$ above is now regular, thus assuring that the Riemann sums converge. It therefore follows from Lemma \ref{riemann} that
	\begin{equation*}
		|D_1|\leq \frac{C}{N^{k+\alpha}}\left\|F\right\|_{C^{k,\alpha}}
		\leq \frac{C}{N^{k+\alpha}}\left\|\gamma^N\right\|_{\ell^1}\left\|x\cot x\right\|_{C^{k,\alpha}\left(\left[0,\frac\pi2\right]\right)}\left\|\varphi'\right\|_{C^{k,\alpha}}
		\leq \frac{C}{N^{k+\alpha}}\left\|\gamma^N\right\|_{\ell^1}\left\|\varphi\right\|_{C^{k+1,\alpha}}.
	\end{equation*}
	Then, further utilizing estimate \eqref{est l1}, Lemma \ref{riemann}, that $k+\alpha\geq\frac 12$ and the fact that $f$ has zero mean value \eqref{flux}, we infer
	\begin{equation}\label{d1}
		\begin{aligned}
			|D_1|
			& \leq \frac{C}{N^{k+\alpha}}
			\left(\left\|f\right\|_{L^\infty}+|\gamma|+{\sqrt N}\left|\frac 1{N-1}\sum_{i=1}^{N-1}f(\tilde\theta_i^N)\right|\right)
			\left\|\varphi\right\|_{C^{k+1,\alpha}}
			\\
			& \leq \frac{C}{N^{k+\alpha}}
			\left(\left\|f\right\|_{L^\infty}+|\gamma|
			+{\sqrt N}\left|\int_0^{2\pi}f(\theta){\rm d}\theta - \frac {2\pi}{N}\sum_{i=1}^{N}f(\tilde\theta_i^N)\right|\right)
			\left\|\varphi\right\|_{C^{k+1,\alpha}}
			\\
			& \leq \frac{C}{N^{k+\alpha}}
			\left(\left\|f\right\|_{C^{k,\alpha}}+|\gamma|
			\right)
			\left\|\varphi\right\|_{C^{k+1,\alpha}}.
		\end{aligned}
	\end{equation}

	Finally, regarding $D_4$, recalling that, by \eqref{point toy} and \eqref{perfect distri},
	\begin{equation*}
\sum_{i=1}^{N-1}f(\tilde \theta_{i}^N) = \frac1N\sum_{j=1}^N \gamma_{j}^N \sum_{i=1}^{N-1} \cot\left(\frac{\tilde\theta_{i}^N-\theta_{j}^N}2\right)
		= -\frac1N\sum_{j=1}^N \gamma_{j}^N \cot\left(\frac{\tilde\theta_{N}^N-\theta_{j}^N}2\right)
		=-f_{\rm app}^N(\tilde\theta_N^N),
	\end{equation*}
	we find
	\begin{equation*}
		\begin{aligned}
			D_4 & =\frac{2\pi}{N} \left( f_{\rm app}^N(\tilde\theta_{N}^N)- f(\tilde\theta_{N}^N) \right) \varphi(\tilde\theta_{N}^N)
			=-\frac{2\pi}{N} \sum_{i=1}^N f(\tilde \theta_{i}^N)\varphi(\tilde\theta_{N}^N) \\
			& =\left(\int_0^{2\pi}f(\theta){\rm d}\theta - \frac {2\pi}{N}\sum_{i=1}^{N}f(\tilde\theta_i^N)\right)\varphi(\tilde\theta_{N}^N).
		\end{aligned}
	\end{equation*}
	Hence, utilizing Lemma \eqref{riemann} again,
	\begin{equation}\label{d4}
		\left|D_4\right|
		\leq \frac{C}{N^{k+\alpha}}\|f \|_{C^{k,\alpha}}\left\|\varphi\right\|_{L^\infty}.
	\end{equation}

	On the whole, since $D_3=0$, combining \eqref{d2}, \eqref{d1} and \eqref{d4}, we deduce that
	\begin{equation*}
		\left|\int_{0}^{2\pi} (f_{\rm app}^N - f )\varphi\right|
		\leq
		\frac{C}{N^{k+\alpha}}
		\left(\left\|f\right\|_{C^{k,\alpha}}+|\gamma|
		\right)
		\left\|\varphi\right\|_{C^{k+1,\alpha}},
	\end{equation*}
	which concludes the proof of the proposition.
\end{proof}

\section{Proof of Theorem \ref{main theo}}

We proceed now to the demonstration of our main result --~Theorem \ref{main theo}~-- on the approximation of the boundary of an exterior domain by point vortices.

First, for given $\omega\in C_c^{0,\alpha}$ and $\gamma\in\mathbb{R}$, recall that the full plane flow $u_{P}\in C^1\left(\overline \Omega\right)$ is obtained from \eqref{uP} and that the $2\pi$-periodic function $f\in C^\infty\left([0,2\pi]\right)$, which has zero mean value \eqref{flux}, is defined by \eqref{f}. Therefore, with this given $f$, according to Corollary \ref{inverse perfect}, we find a unique solution $\gamma^N\in\mathbb{R}^N$ of \eqref{point toy}.

Next, the approximate flow $u_{\rm app}^N$ is introduced by \eqref{approx}, which verifies by \eqref{cot}:
\[
u_{\rm app}^N(x)\cdot n(x) = - \frac1{4\pi} f_{\rm app}^N(\theta),
\]
where $x=(\cos\theta,\sin\theta)\in\partial B(0,1)$ and $f_{\rm app}^N$ is defined by \eqref{f app}.
Utilizing identity \eqref{vortex identity} to rewrite the discrete Biot-Savart kernel of $u_{\rm app}^N$, we find that
\begin{align*}
	u_{\rm app}^N(x)& = \frac1{2\pi} \sum_{j=1}^N \frac{\gamma_{j}^N}N 
	\left(
		\frac{1}{2\pi}\int_{\partial B(0,1)}
		\frac{x-z}{|x-z|^2}
		\cot\left(\frac{\tilde\theta_{j}^N-\theta}{2}\right) dz
		+\frac{x^\perp}{|x|^2}
	\right)\\
	&=\frac{-1}{4\pi^2}\int_{\partial B(0,1)} \frac{x-z}{|x-z|^2}f_{\rm app}^N(z)dz + \frac \gamma{2\pi}\frac{x^\perp}{|x|^2}, \quad\text{on }\Omega.
\end{align*}

Furthermore, recall that, according to \eqref{boundary sheet circle 2}, the remainder flow $u_R$ can be expressed as
\begin{equation*}
	u_R(x) = -\frac 1{4\pi^2}\int_{\partial B(0,1)} \frac{x-y}{|x-y|^2} f(y)dy
	+ \frac \gamma{2\pi} \frac{x^\perp}{|x|^2}, \quad\text{on }\Omega,
\end{equation*}
whereby
\begin{equation*}
	\left(u_R-u_{\rm app}^N\right)(x) = \frac 1{4\pi^2}\int_{\partial B(0,1)} \frac{x-y}{|x-y|^2} \left(f^N_{\rm app}-f\right)(y)dy, \quad\text{on }\Omega.
\end{equation*}

Therefore, in view of Proposition \ref{prop 32}, we deduce that, for any fixed $x\in\Omega$,
\begin{equation*}
	\left|\left(u_R-u_{\rm app}^N\right)(x)\right|\leq \frac{C}{N^2}\left\|\frac{x-y}{|x-y|^2}\right\|_{C_y^3}
	\leq \frac{C}{N^2}\sup_{y\in\partial B(0,1)}\left(\frac{1}{|x-y|}+\frac{1}{|x-y|^4}\right),
\end{equation*}
where the constant $C>0$ only depends on $\omega$ and $\gamma$. It follows that, for any closed set $K\subset\Omega$,
\begin{equation*}
	\| u_{R} - u_{\rm app}^N \|_{L^\infty(K)} \leq \frac{C}{N^2},
\end{equation*}
which concludes the proof of the theorem. \qed

\bigskip

\noindent
 {\bf Acknowledgements.} The authors are partially supported by the project \emph{Instabilities in Hydrodynamics} funded by the Paris city hall (program \emph{Emergences}) and the \emph{Fondation Sciences Math\'ematiques de Paris}. C.L. is partially supported by the Agence Nationale de la Recherche, Project DYFICOLTI, grant ANR-13-BS01-0003-01.

\def\cprime{$'$}

\end{document}